\documentclass[reqno,oneside]{amsart}
\usepackage[hidelinks]{hyperref}
\usepackage{geometry}
\usepackage[ansinew]{inputenc}
\usepackage{graphicx}
\usepackage{amsmath}% amstex ?
\usepackage{amsthm}
\usepackage{amssymb, color}%
\usepackage{mathrsfs}
\usepackage{bbm}
\usepackage{mathptmx}

\newcommand{\R}{\mathbb{R}}
\newcommand{\E}{\mathbb{E}}
\newcommand{\N}{\mathbb{N}}

\newcommand{\Z}{\mathbb{Z}}

\newcommand{\eee}{{\rm e}}

\newcommand{\eqdistr}{\stackrel{d}{=}}
\newcommand{\todistr}{\Longrightarrow}%{\stackrel{{\rm w}}{\longrightarrow}}

\newcommand{\1}{\mathbbm{1}}
\newcommand{\be}{\begin{equation}}
\newcommand{\ee}{\end{equation}}
\newcommand{\HG}{\operatorname{HyperG}\,}
\newcommand{\Bin}{\operatorname{Bin}\,}
\newcommand{\BETA}{\operatorname{Beta}\,}
\newcommand{\Normal}{\mathcal N}

\theoremstyle{plain}
\newtheorem{theorem}{Theorem}[section]
\newtheorem{lemma}[theorem]{Lemma}
\newtheorem{corollary}[theorem]{Corollary}
\newtheorem{proposition}[theorem]{Proposition}
\theoremstyle{definition}

\theoremstyle{remark}
\newtheorem{remark}[theorem]{Remark}

\begin{document}
\title{When does the chaos in the Curie-Weiss model stop to propagate?}

\begin{abstract}
We investigate increasing propagation of chaos for the mean-field Ising model of ferromagnetism (also known as the Curie-Weiss model) with $N$ spins at inverse temperature $\beta>0$ and subject to an external magnetic field of strength $h\in\mathbb{R}$.
Using a different proof technique than in \cite{BAZ_chaos} we confirm the well-known propagation of chaos phenomenon: 
If $k=k(N)=o(N)$ as $N\to\infty$, then the $k$'th marginal distribution of the Gibbs measure converges to a product measure at $\beta <1$ or $h \neq 0$ and to a mixture of two product measures, if $\beta >1$ and $h =0$.
More importantly, we also show that if $k(N)/N\to \alpha\in (0,1]$, this property is lost and we identify a non-zero limit of the total variation distance between the number of positive spins among any $k$-tuple and the corresponding binomial distribution.
\end{abstract}

\author[Jonas Jalowy]{Jonas Jalowy}
\address[Jonas Jalowy]{Fachbereich Mathematik und Informatik,
Universit\"at M\"unster,
Einsteinstra\ss e 62,
48149 M\"unster,
Germany. Email: {\tt jjalowy@uni-muenster.de}}

\author[Zakhar Kabluchko]{Zakhar Kabluchko}
\address[Zakhar Kabluchko]{Fachbereich Mathematik und Informatik,
Universit\"at M\"unster,
Einsteinstra\ss e 62,
48149 M\"unster,
Germany. Email: {\tt zakhar.kabluchko@uni-muenster.de}}

\author[Matthias L\"owe]{Matthias L\"owe}
\address[Matthias L\"owe]{Fachbereich Mathematik und Informatik,
Universit\"at M\"unster,
Einsteinstra\ss e 62,
48149 M\"unster,
Germany. Email: {\tt maloewe@math.uni-muenster.de}}

\author[Alexander Marynych]{Alexander Marynych}
\address[Alexander Marynych]{Faculty of Computer Science and Cybernetics,
Taras Shevchenko National University of Kyiv,
Volodymyrska 60,
01033 Kyiv,
Ukraine. Email: {\tt marynych@unicyb.kiev.ua}}

\thanks{JJ, ZK and ML are funded by the Deutsche Forschungsgemeinschaft (DFG, German Research Foundation) under Germany's Excellence Strategy EXC 2044 - 390685587, Mathematics M\"unster: \emph{Dynamics-Geometry-Structure}. JJ and ZK have been supported by the DFG priority program SPP 2265 \emph{Random Geometric Systems}.}

\maketitle

\section{Introduction}
The Curie-Weiss model is a mean-field model for ferromagnetism in statistical mechanics.
It is described by a sequence of probability measures, the Gibbs measures $\mu_N$, on the sets $\{-1,+1\}^N$.
These measures are parametrized by a positive parameter $\beta>0$ known
as the inverse temperature and a real parameter $h\in\mathbb{R}$ interpreted as strength of an external magnetic field. Given such $\beta>0$ and $h\in\mathbb{R}$, the Gibbs measure takes the following form:
\be \label{eq:Gibbs}
\mu_N(\sigma):=\mu_{N,\beta,h}(\sigma):= \frac{\exp\left(\frac{\beta}{2N} \sum_{i,j=1}^N \sigma_i \sigma_j+h\sum_{i=1}^{N}\sigma_i\right)}{Z_N} \qquad \sigma:= (\sigma_i)_{i=1}^N\in \{-1,+1\}^N.
\ee
The normalizing constant
\be
Z_N=Z_N(\beta,h)=\sum_{\sigma' \in \{-1,+1\}^N} \exp\left(\frac{\beta}{2N} \sum_{i,j=1}^N \sigma_i' \sigma_j'+h\sum_{i=1}^{N}\sigma_i'\right)
\ee
is called the partition function of the model.

There is a vast literature on the Curie-Weiss model with the main asymptotic results summarized in the textbooks~\cite{BovierSMoDS}, \cite{EllisEntropyLargeDeviationsAndStatisticalMechanics} and \cite{Velenik_book}; see also the papers~\cite{Chatterjee_Shao,EL10,EllisNewman_80}. The order parameter of the model is called the magnetization and is defined by
$$
m_N:=m_N(\sigma) := \frac 1N \sum_{i=1}^N \sigma_i=\frac{2}{N}\mathcal{P}_N-1,
$$
where $\mathcal{P}_N:=\mathcal{P}_N(\sigma):=|\{i=1,\ldots,N:\sigma_i=+1\}|$ is the number of positive spins.  Let $\mu_N\circ m_N^{-1}$ denote the distribution of the random variable $m_N$ under the probability measure $\mu_N$.  The first-order limiting behavior of the magnetization  is given by
\begin{align}\label{eq:mconvergence}
\mu_N\circ m_N^{-1} \Rightarrow
\begin{cases}\delta_{{\tt m}(\beta,h)},& \text{if}\quad h\neq 0\quad\text{or}\quad 0<\beta \le 1,\\ \frac{1}{2}\left(\delta_{{\tt m}(\beta,0)}+ \delta_{-{\tt m}(\beta,0)}\right),& \text{if}\quad h=0\quad\text{and}\quad \beta>1,\\ \end{cases}
%(\mu_N\circ m_N^{-1})(\cdot)=\mu_N(\{\sigma : m_N(\sigma)\in \cdot \}) \todistr \begin{cases}\delta_{{\tt m}(\beta,h)}(\cdot),& \text{if}\quad h\neq 0\quad\text{or}\quad 0<\beta<1,\\ \frac{1}{2}\left(\delta_{{\tt m}(\beta,0)}(\cdot)+ \delta_{-{\tt m}(\beta,0)}(\cdot)\right),& \text{if}\quad h=0\quad\text{and}\quad \beta>1,\\ \end{cases}
\end{align}
which indicates a phase transition at $\beta=1$ in the absence of the external magnetic field ($h=0$). Here, $\todistr$ denotes weak convergence, $\delta_x$ is the Dirac measure at $x$ and ${\tt m}(\beta,h)$ is the largest in absolute value solution to
\begin{align}\label{eq:tanhh}
z=\tanh(\beta z+h).
\end{align}
This solution is unique and positive if $h>0$, unique and negative if $h<0$, and it is equal to zero if $h=0$ and $0<\beta  \leq 1$. If $h=0$ and $\beta>1$, equation~\eqref{eq:tanhh} has two non-zero solutions ${\tt m}(\beta,0)$ and $-{\tt m}(\beta,0)$.

For the purpose of the present paper it is important that the  weak convergence~\eqref{eq:mconvergence} is accompanied by the corresponding Gaussian approximations whenever $\beta\neq 1$ or $h\neq 0$. Let us briefly summarize these, postponing the rigorous statements to Section~\ref{sec:clt_llt_magnetization} below. By definition, $|{\tt m}(\beta,h)|<1$ and it is also true that $0<\beta(1-{\tt m}^2(\beta,h))<1$ for all $\beta>0$, $h\in \R$ excluding the ``critical'' case $(\beta, h) = (1,0)$.
%for all $(\beta,h)\in ((0,\infty)\times (\mathbb{R}\setminus\{0\}))\cup (((0,1)\cup(1,\infty))\times \mathbb{R})$.
Put
\begin{equation}\label{eq:v_beta_h_def}
{\tt v}^2_{\beta,h}:=\frac{1-{\tt m}^2(\beta,h)}{1-\beta(1-{\tt m}^2(\beta,h))},\quad
(\beta, h) \neq (1,0).
%(\beta,h)\in ((0,\infty)\times (\mathbb{R}\setminus\{0\}))\cup (((0,1)\cup(1,\infty))\times \mathbb{R}).
\end{equation}
If $ 0 <  \beta < 1$ and $h=0$, then $\sqrt N m_N$ has Gaussian fluctuations under $\mu_N$ with expectation ${\tt m}(\beta,0)=0$ and variance ${\tt v}^2_{\beta,0}=1/(1-\beta)$. Of course, this cannot be true for $\beta=1$, $h=0$. In this case, the distribution of $N^{1/4} m_N$ has a limiting density
\be
f_1(x)=\frac{\exp(-\frac 1{12} x^4)}{\int_\R \exp(-\frac 1{12} y^4)dy}.\ee
For $\beta > 1$, still assuming $h=0$, there is a conditional central limit theorem for $\sqrt N(m_N-{\tt m}(\beta,0))$ (respectively $\sqrt N(m_N+{\tt m}(\beta,0))$, conditioned on
$m_N>0$ (respectively, $m_N<0$). In this case the limiting expectation is $0$ and the limiting variance is ${\tt v}^2_{\beta,0}$. Finally, if $h\neq 0$, there are no phase transitions as $\beta$ varies in $(0,
\infty)$ and  $\sqrt N(m_N-{\tt m}(\beta,h))$ converges to the centred normal distribution with variance ${\tt v}^2_{\beta,h}$.
%$m_N$ is asymptotically normal with mean ${\tt m}(\beta,h)$ and variance ${\tt v}^2_{\beta,h}$.  {\color{red} $\sqrt N$ missing?}

\subsection{Propagation of chaos and the main results}

For the time being, fix $k\in\mathbb{N}$ and pick any $k$-tuple among $N$ spins.
The propagation of chaos paradigm for Gibbs measure states that for a
mean-field model as the Curie-Weiss model the marginal distributions of the $k$ spins become asymptotically independent. This shall
be investigated in the sequel.
Since the family of random variables $(\sigma_i)_{i=1}^N$ is exchangeable under the Gibbs measure $\mu_N$, without loss of generality we
may pick the first $k$ spins and consider their marginal distribution $\mu^{(k)}_{N,\beta,h}$.
Let $\mathcal{P}_{k}:=|\{j\in\{1, \ldots,k\}:\sigma_j=+1\}|$ be the number of positive spins among the picked ones. Note that
$\mathcal{P}_{k}$ completely determines $\mu^{(k)}_{N,\beta,h}$, hence we might as well study the distribution of $\mathcal{P}_{k}$
under $\mu_N$,  denoted by $\mu_N\circ \mathcal{P}_{k}^{-1}$.
Intuitively, if $h=0$ and $ 0 < \beta<1$, as $N\to\infty$, the first $k$ spins should indeed be asymptotically independent and
take values $\pm 1$ with the same probability $1/2$.
This implies that the distribution of $\mathcal{P}_{k}$ should be close to the binomial distribution with parameters $k$ and $1/2$.
%This observation is a particular instance of a general phenomenon called 'propagation of chaos' and
This can formally be written as follows. Let $\Bin(n,p)$ denote a binomial distribution with parameters $n\in\N_0$ and $p\in [0,1]$:
\begin{equation}\label{eq:binomial_def}
\Bin(n,p)(\{j\})=\binom{n}{j}p^{j}(1-p)^{n-j},\quad 0\leq j\leq n.
\end{equation}
Throughout the paper we shall slightly abuse the notation and write mixtures of distributions by simply writing a mixing distribution instead of a parameter which is being mixed. For example, for any distribution $\mathcal L(W)$ of a random variable $W$ on $[0,1]$, the distribution $\Bin(n,\mathcal L(W))$ should be understood as
$$
\Bin(n,\mathcal L (W))(\{j\})=\E\big(\Bin(n,W)(\{j\})\big).
$$
Recall that the total variation distance $d_{TV}$ between two probability measures $\mathfrak{M}_1$ and $\mathfrak{M}_2$ on $\R$ is defined by
$$
d_{TV}(\mathfrak{M}_1,\mathfrak{M}_2)=\sup_{A\in\mathcal{B}(\R)}|\mathfrak{M}_1(A)-\mathfrak{M}_2(A)|,
$$
where $\mathcal{B}(\R)$ is the Borel sigma-algebra. Recall also that $k\in \N$ is fixed for the time being. The 'propagation of chaos' phenomenon for the Curie-Weiss model tells us that, if $h=0$  and $0<\beta<1$, then
\begin{equation}\label{eq:chaos_prop_basic}
\lim_{N\to\infty}d_{TV}(\mu_N\circ \mathcal{P}_{k}^{-1},\Bin(k,1/2))=0.
\end{equation}
%{\color{red}}
On the other hand, for $h= 0$ and $\beta>1$, it is known that
\begin{equation}\label{eq:chaos_prop_basic_beta_greater_1}
\lim_{N\to\infty}d_{TV}\big(\mu_N\circ \mathcal{P}_{k}^{-1},\Bin\big(k,\frac{1}{2}(\delta_{(1+{\tt m}(\beta,0))/2}+\delta_{(1-{\tt m}(\beta,0))/2}\big)\big)=0.
\end{equation}

Note that the approximating distribution is a mixture of two binomial distributions, one with success probability $(1+{\tt m}(\beta,0))/2$ and the other one with $(1-{\tt m}(\beta,0))/2$.
Finally, it is known that for $h\neq 0$  (and arbitrary $\beta>0$), relation~\eqref{eq:chaos_prop_basic} holds with $1/2$ replaced by $(1+{\tt m}(\beta,h))/2$.

Assume now that $k=k(N)$ depends on $N$ and $\lim_{N\to\infty}k(N)=\infty$.  The aim of the this note is to answer the question whether the analogues of~\eqref{eq:chaos_prop_basic} and~\eqref{eq:chaos_prop_basic_beta_greater_1} hold true in this case and if not, what is a threshold on the growth of $k(N)$ such that the limit in~\eqref{eq:chaos_prop_basic} or~\eqref{eq:chaos_prop_basic_beta_greater_1}
becomes non-zero and what is the value of the limit.
Note that for $k(N)/N\to 0$, that is $k=o(N)$, propagation of chaos has been explicitly shown to hold for $\beta \neq 1$ and $h=0$ in
\cite{BAZ_chaos} (note that the authors consider relative entropy rather than total variation distance and that their model also covers $h\neq 0$ implicitly).
Our answer is provided by Theorem~\ref{thm:main1} below, which is our main result.
To simplify its formulation, let us introduce additional notation.

For ${\tt m}\in\R$ and ${\tt v}^2>0$ put
\begin{equation}\label{eq:normal density}
\varphi(t;{\tt m},{\tt v}^2):=\frac{1}{{\tt v}\sqrt{2\pi}}\exp{\left(-\frac{(t-{\tt m})^2}{2{\tt v}^2}\right)},\quad t\in\mathbb{R}.
\end{equation}
That is to say,  $t\mapsto \varphi(t;{\tt m},{\tt v}^2)$ is the density of a Gaussian distribution with mean
${\tt m}$ and variance ${\tt v}^2$. Define
\begin{equation}\label{eq:def_D_dist_normal_distr}
D({\tt v}_1^2,{\tt v}_2^2):=\frac{1}{2}\int_{\mathbb{R}}\left|\varphi(t;0,{\tt v}^2_1)-\varphi(t;0,{\tt v}^2_2)\right|{\rm d}t
\end{equation}
to be the total variation distance between two centred Gaussian distributions, which can easily be calculated in terms of the error function. Then our main result reads as follows.

\begin{theorem}\label{thm:main1}
Let $k=k(N)$ be any sequence of positive integers such that
\begin{equation}\label{eq:K_N_linear}
\lim_{N\to\infty} k(N)=+\infty,\quad\lim_{N\to\infty}\frac{k(N)}{N}=\alpha \in [0,1].
\end{equation}
If $h=0$ and $\beta\in (0,1)$, then
\begin{equation}\label{eq:thm1_claim1}
d_{TV}(\mu_N\circ \mathcal{P}_{k(N)}^{-1},\Bin(k(N),1/2))=D\left(\frac{1-\beta+\alpha\beta}{4(1-\beta)},\frac{1}{4}\right).
\end{equation}
If $h=0$ and $\beta>1$, then
\begin{multline}\label{eq:thm1_claim2}
d_{TV}\left(\mu_N\circ \mathcal{P}_{k(N)}^{-1},\Bin\left(k(N),\frac{1}{2}(\delta_{(1+{\tt m}(\beta,0))/2}+\delta_{(1-{\tt m}(\beta,0))/2})\right)\right)\\
=D\left(\frac{1-{\tt m}^2(\beta,0)}{4}\left(1+\frac{\alpha\beta(1-{\tt m}^2(\beta,0))}{1-\beta(1-{\tt m}^2(\beta,0))}\right),\frac{1-{\tt m}^2(\beta,0)}{4}\right).
\end{multline}
If $h\neq 0$ and $\beta>0$, then
\begin{multline}\label{eq:thm1_claim3}
d_{TV}\big(\mu_N\circ \mathcal{P}_{k(N)}^{-1},\Bin(k(N),(1+{\tt m}(\beta,h))/2)\big)\\
=D\left(\frac{1-{\tt m}^2(\beta,h)}{4}\left(1+\frac{\alpha\beta(1-{\tt m}^2(\beta,h))}{1-\beta(1-{\tt m}^2(\beta,h))}\right),\frac{1-{\tt m}^2(\beta,h)}{4}\right).
\end{multline}

In particular, if $k(N)=o(N)$, then $\alpha=0$ and the right-hand sides in~\eqref{eq:thm1_claim1}, \eqref{eq:thm1_claim2} and~\eqref{eq:thm1_claim3} vanish, meaning that the 'propagation of chaos' persists for any $k(N)$ of a sublinear growth.
\end{theorem}

 Let us now informally discuss the case when $\alpha > 0$. For simplicity, we consider~\eqref{eq:thm1_claim1}. The limit on the right-hand side is non-zero, which suggests that there is a residual dependence between the $k(N)$ spins under the Gibbs measure. The reason for the non-zero limit is the fact that the distribution of $\mathcal P_{k(N)}$ and the corresponding binomial distribution satisfy  central limit theorems with different variances, the variance of $\mathcal P_{k(N)}$ being strictly larger, which comes from the fact that the spins are positively correlated under the Gibbs measure.  The distance between these normal distributions appears on the right-hand side of~\eqref{eq:thm1_claim1}.
%In order to prove Theorem~\ref{thm:main1} we shall determine 
In Theorem~\ref{thm:P_Nk_approx_h_neq_0-beta-in-0-1}, we shall determine a \textit{mixed} binomial distribution which approximates the distribution of $\mathcal P_{k(N)}$ under $\mu_N$. In some sense, this describes  the residual dependence between the spins under the Gibbs measure.  %For~\eqref{eq:thm1_claim2} and~\eqref{eq:thm1_claim3}, similar approximations by mixed distributions   will be established in Theorems~\ref{thm:P_Nk_approx_h_neq_0-beta-in-0-1} and~\ref{thm:P_Nk_approx_h=0-beta>1}.

%In particular, this mixed binomial distribution has a larger variance (given as the first argument of $D$) compared to the binomial distribution of `independent spins' (given as the second argument of $D$), }

\begin{remark}
The exchangeability of the measure $\mu_N$ has been used to investigate the Curie-Weiss model for example, in~\cite[Section~5.2]{kirsch_survey:2015} and~\cite{barhoumi_butzek_eichelsbacher:2023}. In particular, an explicit representation of $\mu_N$ as a mixture of Bernoulli measures (valid for each fixed $N$) can be found in~\cite[Theorem~5.6]{kirsch_survey:2015}. A general propagation of chaos principle stating that the distribution of $k$ entries in a finite exchangeable vector of length $n$ can be approximated by a mixture of i.i.d.\ distributions can found in~\cite{diaconis_freedman:1980}.
\end{remark}

The paper is organized as follows. Our proof relies on local limit theorems for the magnetization $m_N$ and also for the total number of positive spins $\mathcal{P}_N$ under $\mu_N$. In some regimes those are known. We collect the corresponding results in Section~\ref{sec:clt_llt_magnetization} below. The proofs of these local limit theorems, which we have not been able to locate in the literature, are given in Section~\ref{sec:proof_magnetization}. The proof of Theorem~\ref{thm:main1} is given in Section~\ref{sec:proof}, including the statement of residual dependence. Two auxiliary technical results related to calculations of the total variation distance are presented in Section~\ref{sec:appendix}.

\section{Local limit theorem for the magnetization}\label{sec:clt_llt_magnetization}

Denote by $\Normal({\tt m},{\tt v}^2)$ a Gaussian distribution with mean ${\tt m}$ and variance ${\tt v}^2$, so
$$
\Normal({\tt m},{\tt v}^2)(A)=\int_{A}\varphi(t;{\tt m},{\tt v}^2){\rm d}t,\quad A\in\mathcal{B}(\R).
$$
Put $\delta_N:=(1-(-1)^N)/2$. This correction term appears below in the local limit theorems for $m_N$, since $Nm_N$ always has the same parity as $N$.

\begin{proposition}\label{prop:clt_llt_h_neq_0-beta-in-0-1}
Assume that $h\neq 0$ or $0<\beta<1$. Then
$$
\mu_N\left(\sqrt{N}(m_N-{\tt m}(\beta,h))\in\cdot\right)~\todistr~\Normal\left(0,{\tt v}^2_{\beta,h}\right),\quad N\to\infty,
$$
and the following local limit theorem holds true:
\begin{equation}\label{eq:clt_llt_h_neq_0-beta-in-0-1}
\lim_{N\to\infty}\sqrt{N}\sup_{\ell\in\mathbb{Z}}\left|\mu_N\left( \frac{Nm_N+\delta_N}{2}=\ell\right)-\varphi\left(\ell;\frac{N{\tt m}(\beta,h)}{2},\frac{N{\tt v}^2_{\beta,h}}{4}\right)\right|=0.
\end{equation}
\end{proposition}

\begin{proposition}\label{prop:clt_llt_h=0-beta>1}
Assume that $h=0$ and $\beta>1$. Then
$$
\mu_N\left(\sqrt{N}(m_N-(+){\tt m}(\beta,0))\in\cdot\;|\;m_N>(<)0\right)~\todistr~\Normal\left(0,{\tt v}^2_{\beta,0}\right),\quad N\to\infty,
%\frac{\mu_N\left(\left\{\sigma:\sqrt{N}(m_N(\sigma)-(+){\tt m}(\beta,0))\in\cdot\;,\;m_N(\sigma)>(<)0\right\}\right)}{\mu_N\left(\left\{\sigma:m_N(\sigma)>(<)0\right\}\right)}~\todistr~\Normal\left(0,{\tt v}^2_{\beta,0}\right),\quad N\to\infty,
$$
and the following local limit theorem holds true:
\begin{multline}\label{eq:clt_llt_h=0-beta>1}
\lim_{N\to\infty}\sqrt{N}\sup_{\ell\in\mathbb{Z}}\Big|\mu_N\left(\frac{Nm_N+\delta_N}{2}=\ell\right)
\\-\frac{1}{2}\left(\varphi\left(\ell;\frac{N{\tt m}(\beta,0)}{2},\frac{N{\tt v}^2_{\beta,0}}{4}\right)+\varphi\left(\ell;-\frac{N{\tt m}(\beta,0)}{2},\frac{N{\tt v}^2_{\beta,0}}{4}\right)\right)\Big|=0.
\end{multline}
\end{proposition}

For some values of $(\beta,h)$ the above local limit theorems can be extracted from the vast literature on the Curie-Weiss model. For
example, in the high-temperature regime $\beta\in (0,1)$
and for every $h\in\mathbb{R}$,~\eqref{eq:clt_llt_h_neq_0-beta-in-0-1} has been proved
in~\cite[Theorem 4.5 and Eq.~(4.4)]{Rollin-Ross:2015}. If $h>0$ and $\beta>0$, then~\eqref{eq:clt_llt_h_neq_0-beta-in-0-1} can be found
in~\cite[Theorem 2.14 and Lemma 1.1]{Barbour+Rollin+Ross:2019}. The missing case $h<0$ and $\beta>0$ in
Proposition~\ref{prop:clt_llt_h_neq_0-beta-in-0-1} can be derived by the same methods. Finally, if $h=0$, a local limit theorem for a multi-group Curie-Weiss model in
the  high-temperature regime $\beta\in(0,1)$ has been derived in~\cite{Fleermann+Kirsch+Toth:2022}. Quite (un-)expectedly, we have not been
able to locate Proposition~\ref{prop:clt_llt_h=0-beta>1} in the literature, because of the non-standard approximation by a mixture of
normal distributions. We shall give an elementary proof based on the Stirling approximation in Section~\ref{sec:proof_magnetization}.

We shall actually need local limit theorems for $\mathcal{P}_N$ rather than $m_N$. They follow immediately from
Propositions~\ref{prop:clt_llt_h_neq_0-beta-in-0-1} and~\ref{prop:clt_llt_h=0-beta>1} using the obvious relation
$
\mathcal{P}_N=\frac{N}{2}(1+m_N)
$
together with the bound
$$
|\varphi(t_1,{\tt m},{\tt v}^2)-\varphi(t_2,{\tt m},{\tt v}^2)|\leq C\frac{|t_2-t_2|}{{\tt v}^2},\quad t_1,t_2\in\mathbb{R},
$$
for some absolute constant $C>0$, which is a consequence of the mean value theorem for differentiable functions. The above bound allows us to neglect the correction term $\delta_N$ appearing in local limit theorems for $m_N$.
\begin{corollary}\label{cor:clt_llt_h_neq_0-beta-in-0-1}
Assume that $h\neq 0$ or $0<\beta<1$. Then
$$
\mu_N\left(N^{-1/2}(\mathcal{P}_N-\tfrac N2(1+{\tt m}(\beta,h)))\in\cdot\right)~\todistr~\Normal\left(0,\tfrac 1 4{\tt v}^2_{\beta,h}\right),\quad N\to\infty,
$$
and the following local  limit theorem holds true
$$
\lim_{N\to\infty}\sqrt{N}\sup_{\ell\in\mathbb{Z}}\left|\mu_N\left(\mathcal{P}_N=\ell\right)-\varphi(\ell;\tfrac N2(1+{\tt m}(\beta,h)),\tfrac N4{\tt v}^2_{\beta,h})\right|=0.
$$
\end{corollary}

\begin{corollary}\label{cor:clt_llt_h=0-beta>1}
Assume that $h=0$ and $\beta>1$. Then
$$
\mu_N\left(N^{-1/2}\big(\mathcal{P}_N-\tfrac N2(1+(-){\tt m}(\beta,0))\big)\in\cdot\;|\;\mathcal{P}_N>(<)\tfrac N2\right)\todistr~\Normal\left(0,\tfrac 1 4{\tt v}^2_{\beta,0}\right),\quad N\to\infty,
%\frac{\mu_N\left(\left\{\sigma:N^{-1/2}(\mathcal{P}_N(\sigma)-N(1-(+){\tt m}(\beta,0))/2)\in\cdot\;,\;\mathcal{P}_N(\sigma)>(<)N/2\right\}\right)}{\mu_N\left(\left\{\sigma:\mathcal{P}_N(\sigma)>(<)N/2\right\}\right)}\\ \todistr~\Normal\left(0,{\tt v}^2_{\beta,0}/4\right),\quad N\to\infty,
$$
and the following local limit theorem holds true
\begin{multline*}
\lim_{N\to\infty}\sqrt{N}\sup_{\ell\in\mathbb{Z}}\Big|\mu_N\left(\mathcal{P}_N=\ell\right)
-\tfrac{1}{2}\left(\varphi\left(\ell;\tfrac{N}2(1+{\tt m}(\beta,0)),\tfrac{N}4 {\tt v}^2_{\beta,0}\right)+\varphi\left(\ell;\tfrac{N}2(1-{\tt m}(\beta,0)),\tfrac{N}4{\tt v}^2_{\beta,0}\right)\right)\Big|=0.
\end{multline*}
\end{corollary}

\section{Proof of Theorem~\ref{thm:main1}}\label{sec:proof}

We embark on a simple observation which is a consequence of exchangeability of the spins under the Gibbs measure $\mu_N$. Given that $\mathcal{P}_N=i\in\N_0$%=|\{i\in\{1, \ldots,N\}:\sigma_i=+1\}|=m$
, the conditional distribution of $\mathcal{P}_k$ is hypergeometric with parameters $N$, $i$ and $k$ denoted hereafter $\HG(N,i,k)$. Recall that
%$\HG(N,m,k)$ is the distribution of the number of white balls in a sample of $k$ balls drawn without replacement from a population of $N$ balls, $m$ of which are white and $N-m$ are black. Thus, for $n\in\N_0$, $m\leq n$ and $k\leq n$:
$$
\HG(n,i,k)(\{j\})=\frac{\binom{i}{j}\binom{n-i}{k-j}}{\binom{n}{k}},\quad 0\leq j\leq \min(i,k),\quad i\leq n.
$$
In other words,  the distribution of $\mathcal{P}_{k(N)}$ can be represented as the following mixture of hypergeometric distributions:
\begin{align}\label{eq:distribution_Pk}
\mu_N\circ\mathcal{P}_{k(N)}^{-1}=\HG(N,\mathcal P_N,k(N)).
\end{align}
The family of hypergeometric distributions possesses the following  property %formulated in Lemma~\ref{lem:hypergeometric_binomial} below,
which is of major importance for us. %Throughout the rest of the paper we shall slightly abuse the notation and write mixtures of distributions by simply writing a mixing distribution instead of a parameter which is being mixed. For example,
%$\HG(n,\Bin(n,p),k)$ in the next lemma should be understood as
%$$
%\HG(n,\Bin(n,p),k)(\{j\})=\sum_{m}\HG(n,m,k)(\{j\})\Bin(n,p)(\{m\}).
%$$
\begin{lemma}\label{lem:hypergeometric_binomial}
For $n\in\N_0, k\le n$ and $p\in [0,1]$, it holds %let $\Bin(n,p)$ be the binomial distribution~\eqref{eq:binomial_def}. Then \begin{equation}\label{eq:hyper_binomial_relation}
$\HG(n,\Bin(n,p),k)=\Bin(k,p).$%,\quad k\leq n,\quad n\in\mathbb{N},\quad p\in [0,1].\end{equation}
\end{lemma}
\begin{proof}
For $0\leq j\leq k$,
\begin{align*}
&\hspace{-1cm}\HG(n,\Bin(n,p),k)(\{j\})\\
&=\sum_{i=0}^{n}\binom{n}{i}p^i(1-p)^{n-i}\HG(n,i,k)(\{j\})=\sum_{i=j}^{n}\binom{n}{i}p^i(1-p)^{n-i}\frac{\binom{i}{j}\binom{n-i}{k-j}}{\binom{n}{k}}\\&=\sum_{i=j}^{n}\binom{k}{j}\binom{n-k}{i-j}p^i(1-p)^{n-i}=\binom{k}{j}p^j(1-p)^{k-j}\sum_{i=j}^{n}\binom{n-k}{i-j}p^{i-j}(1-p)^{n-k-i+j}\\
&=\binom{k}{j}p^j(1-p)^{k-j}\sum_{i=0}^{n-j}\binom{n-k}{i}p^{i}(1-p)^{n-k-i}=\binom{k}{j}p^j(1-p)^{k-j}=\Bin(k,p)(\{j\}).
\end{align*}
 Alternatively, we can argue probabilistically: If each of $n$ balls is colored black or white with probability $p$ and $1-p$, respectivelly, and then a sample of $k$ balls is drawn at random from $n$ balls, then the number of black balls in the sample has binomial distribution with parameters $(k,p)$.
\end{proof}

The subsequent proof of Theorem~\ref{thm:main1} proceeds according to the following scheme:

\vspace{2mm}
\noindent
{\sc Step 1.} We approximate $\mathcal{L}(\mathcal{P}_N)$ by an appropriate \emph{mixed} binomial distribution $\Bin(N,\mathcal{L}(W))$, where $W$ is a random variable taking values in $[0,1]$ and $\mathcal{L}(X)$ denotes the distribution of a random variable $X$.  The approximation is understood in the sense of the $d_{TV}$-distance which, as we shall show, converges to $0$.  To accomplish this step we employ the local limit theorems for $\mathcal{P}_N$ provided by Corollaries~\ref{cor:clt_llt_h_neq_0-beta-in-0-1} and~\ref{cor:clt_llt_h=0-beta>1}. It turns out that in a role of the mixing distribution $W$ we can take a beta distribution (or a mixture of two beta-distributions) with properly adjusted parameters.

\vspace{2mm}
\noindent
{\sc Step 2.} Lemma~\ref{lem:hypergeometric_binomial} implies that the $\mu_N$-distribution of $\mathcal{P}_k$ is close (in a sense of the $d_{TV}$-distance) to the mixed binomial distribution
$\Bin(k,\mathcal{L}(W))$. Formal verification of this employs the well-known characterization of the $d_{TV}$-distance
\begin{equation}\label{eq:dtv_coupling}
d_{TV}(\mathfrak{M}_1,\mathfrak{M}_2)=\inf\;\mathbb{P}\{X\neq Y\},
\end{equation}
where the infimum is taken over all pairs $(X,Y)$ of random variables such that $X$ is distributed according to $\mathfrak{M}_1$ and $Y$ is distributed according to $\mathfrak{M}_2$.

\vspace{2mm}
\noindent
{\sc Step 3.} We derive a local limit theorem for $\Bin(k,\mathcal{L}(W))$.

\vspace{2mm}
\noindent
{\sc Step 4.} We calculate the total variation distance between $\Bin(k,\mathcal{L}(W))$ and the three binomial distributions appearing in Theorem~\ref{thm:main1} by using local limit theorems for binomial distributions in conjunction with another well-known formula for $d_{TV}$: If measures $\mathfrak{M}_1$ and $\mathfrak{M}_2$ are supported on $\Z$, then
\begin{equation}\label{eq:dtv_sum}
d_{TV}(\mathfrak{M}_1,\mathfrak{M}_2)=\frac{1}{2}\sum_{k\in\mathbb{Z}}|\mathfrak{M}_1(\{k\})-\mathfrak{M}_2(\{k\})|,
\end{equation}
see Propositions~\ref{prop:diff_variance} and Proposition~\ref{prop:diff_variance_mixture} below.

Our implementation of Steps 1-3 relies on the next proposition. For $\alpha,\beta>0$, $\BETA(\alpha,\beta)$ denotes a beta distribution with the density
$$
\BETA(\alpha,\beta)({\rm d}x)=\frac{x^{\alpha-1}(1-x)^{\beta-1}{\rm d}x}{B(\alpha,\beta)}\1_{\{x\in (0,1)\}},
$$
where $B$ is the Euler beta-function. In what follows we assume that all auxiliary random variables are defined on some probability space $(\Omega,\mathcal{F},\mathbb{P})$.

\begin{proposition}\label{prop:to_zero_in_tv}
Assume that $(\Theta_N)_{N\in\mathbb{N}}$ is a sequence of $\N_0$-valued random variables that satisfy a local limit theorem of the form: for a fixed integer $K\in\mathbb{N}$, a collection of positive weights $p_1,\ldots,p_K$ satisfying $\sum_{j=1}^{K}p_j=1$, $a_j\in (0,1)$ and $\sigma^2_j\in (a_j(1-a_j),\infty)$, $j=1,\ldots,K$, it holds
\begin{equation}\label{eq:lclt_assumption}
\lim_{N\to\infty}\sqrt{N}\sup_{\ell\in\mathbb{Z}}\left|\mathbb{P}(\Theta_N=\ell)-\sum_{j=1}^{K}p_j\varphi(\ell;Na_j,N\sigma^2_j)\right|=0.
\end{equation}
Suppose that $k=k(N)$ is a sequence of positive integers such that~\eqref{eq:K_N_linear} holds for some $\alpha\in [0,1]$ and put
\begin{equation}\label{eq:gammas_def}
\gamma_{j,1}:=\frac{a_j^2(1-a_j)}{\sigma^2_j-a_j(1-a_j)},\quad \gamma_{j,2}:=\frac{a_j(1-a_j)^2}{\sigma^2_j-a_j(1-a_j)},\quad j=1,\ldots,K.
\end{equation}
For $N\in\mathbb{N}$, let $X_N$ be a random variable with a mixed hypergeometric distribution $\HG(N,\mathcal{L}(\Theta_N),k(N))$ and $Y_{N,k(N)}$ be a random variable with the mixed binomial distribution $\Bin(k(N),\sum_{j=1}^{K}p_j\BETA(\gamma_{j,1} N,\gamma_{j,2} N))$, that is,
$$
\mathbb{P}(Y_{N,k(N)}=\ell)= \int_0^1 \binom{k(N)}{\ell}x^{\ell}(1-x)^{k(N)-\ell}\left(\sum_{j=1}^{K}p_j\BETA(\gamma_{j,1} N,\gamma_{j,2} N)({\rm d} x)\right),\quad \ell=0,\ldots,k(N).
$$
Then $\lim_{N\to\infty}d_{TV}(\mathcal{L}(X_N),\mathcal{L}(Y_{N,k(N)}))=0$.
\end{proposition}

\begin{remark}
We shall use this proposition with $\Theta_N=\mathcal{P}_N$ and $K=1$, $p_1=1$ (in conjunction with Corollary~\ref{cor:clt_llt_h_neq_0-beta-in-0-1}) or $K=2$, $p_1=p_2=1/2$ (in conjunction with Corollary~\ref{cor:clt_llt_h=0-beta>1}).
\end{remark}
\begin{remark}\label{rem:blow_up_variance}
Let us provide an informal explanation of Proposition~\ref{prop:to_zero_in_tv}. Consider for simplicity the case $K=1$, $p=1$. Then, \eqref{eq:lclt_assumption} states that $\Theta_N$ satisfies a local  limit theorem with asymptotic centering $Na_1$. One is therefore tempted to approximate $\Theta_N$ by the binomial distribution $\Bin (N, a_1)$, however the variance of this distribution, which equals $Na_1 (1-a_1)$, is strictly smaller than the asymptotic variance $N\sigma_1^2$ appearing in~\eqref{eq:lclt_assumption} due to the assumption $\sigma_1^2 > a_1(1-a_1)$. Instead, we approximate $\Theta_N$ by a \emph{mixed} binomial distribution that artificially blows up the variance until the variances  match.  The informal explanation of this is that both distributions satisfy a local limit theorem with the same centering and normalization.
In the case $k(N) = N$, we have $\Theta_N\eqdistr X_N$ and Proposition~\ref{prop:to_zero_in_tv} states that the total variation distance between the distribution of $X_N$ and $\Bin(N,\BETA(\gamma_{1,1} N,\gamma_{1,2} N))$ converges to $0$.
Moreover, in the formal limit $\sigma_1^2\downarrow a_1(1-a_1)$, we retrieve Lemma \ref{lem:hypergeometric_binomial}.
\end{remark}

\begin{proof}[Proof of Proposition~\ref{prop:to_zero_in_tv}]
We start by noting that, for fixed $j=1,\ldots,K$, the pair $(\gamma_{j,1},\gamma_{j,2})$ is a unique solution to the following system of equations
\begin{equation}\label{eq:system}
		\begin{cases}
		N\frac{\gamma_{j,1}}{\gamma_{j,1}+\gamma_{j,2}}=Na_j,\\
		N\left(\frac{\gamma_{j,1}\gamma_{j,2}}{(\gamma_{j,1}+\gamma_{j,2})^2}+\frac{\gamma_{j,1}\gamma_{j,2}}{(\gamma_{j,1}+\gamma_{j,2})^3}\right)=N\sigma^2_j.
		\end{cases}
\end{equation}
On the left-hand side of the first equation we recognize the mean of $Y_{N,N}$, whereas the left-hand side of the second equation is equal to the variance of $Y_{N,N}$ up to a term $O(1)$. This suggests that the distribution of $\Theta_N$ is close to the distribution of $Y_{N,N}$. In fact, assume that we have proved
\begin{equation}\label{eq:prop_eq1}
\lim_{N\to\infty}d_{TV}(\mathcal{L}(\Theta_N),\mathcal{L}(Y_{N,N}))=0.
\end{equation}
According to~\eqref{eq:dtv_coupling} there exists a sequence of pairs $(\Theta'_N,Y'_{N,N})$ such that
$$
\lim_{N\to\infty}\mathbb{P}(\Theta'_N\neq Y'_{N,N})=0
$$
and $\Theta'_N$ (respectively, $Y'_{N,N}$) has the same distribution as $\Theta_N$ (respectively, $Y_{N,N}$), for every $N\in\mathbb{N}$. Therefore,
\begin{align*}
&d_{TV}(\mathcal{L}(X_N),\HG(N,\mathcal{L}(Y_{N,N}),k(N)))\\
=&d_{TV}(\HG(N,\mathcal{L}(\Theta_N),k(N)),\HG(N,\mathcal{L}(Y_{N,N}),k(N)))\\
=&d_{TV}(\HG(N,\mathcal{L}(\Theta'_N),k(N)),\HG(N,\mathcal{L}(Y'_{N,N}),k(N))) \to 0.%,\quad N\to\infty.
\end{align*}
But this immediately yields the claim since $\HG(N,\mathcal{L}(Y_{N,N}),k(N))$ has the same distribution as $Y_{N,k(N)}$ by Lemma~\ref{lem:hypergeometric_binomial}. Thus, it remains to prove~\eqref{eq:prop_eq1}.

To this end, we shall check that $Y_{N,N}$ satisfies exactly the same local limit theorem as $\Theta_N$, that is,
\begin{equation}\label{eq:lclt_Y_NN}
\lim_{N\to\infty}\sqrt{N}\sup_{\ell\in\mathbb{Z}}\left|\mathbb{P}(Y_{N,N}=\ell)-\sum_{j=1}^{K}p_j\varphi(\ell;Na_j,N\sigma^2_j)\right|=0.
\end{equation}
We shall actually prove a stronger result for later considerations, namely
\begin{equation}\label{eq:lclt_Y_Nk}
\lim_{N\to\infty}\sup_{\ell\in\mathbb{Z}}\left|\sqrt{k(N)}\mathbb{P}(Y_{N,k(N)}=\ell)-\sum_{j=1}^{K}p_j\varphi(\ell;k(N)a_j,k(N)\sigma^2_{\alpha,j})\right|=0,
\end{equation}
where
\begin{equation}\label{eq:variance_alpha}
\sigma^2_{\alpha,j}:=\frac{\gamma_{j,1} \gamma_{j,2}}{(\gamma_{j,1}+\gamma_{j,2})^2}+\frac{\alpha\gamma_{j,1} \gamma_{j,2}}{(\gamma_{j,1}+\gamma_{j,2})^3}.
\end{equation}
Relation~\eqref{eq:lclt_Y_NN} follows from~\eqref{eq:lclt_Y_Nk} upon setting $k(N)=N$. The intuitive fact that~\eqref{eq:lclt_Y_NN} and~\eqref{eq:lclt_assumption} imply~\eqref{eq:prop_eq1} follows from Proposition~\ref{prop:diff_variance_mixture} below.

By the very definition of $Y_{N,k(N)}$ it suffices to prove~\eqref{eq:lclt_Y_Nk} only in the case $K=1$, $p_1=1$ and $j=1$. We find it instructive to first prove a central limit theorem for $Y_{N,k(N)}$ of the form
\begin{equation}\label{eq:Y_n_clt}
\frac{Y_{N,k(N)}-a_1 k(N)}{\sqrt{k(N)}}~\todistr~\Normal(0,\sigma^2_{\alpha,1}),
\end{equation}
which explains the formula for the variance $\sigma^2_{\alpha,1}$ in~\eqref{eq:lclt_Y_Nk}.
Let $B_{N}$ be a random variable with the beta-distribution $\BETA(\gamma_{1,1} N, \gamma_{1,2} N)$ and $C_{N}^{(1)}$, $C_{N}^{(2)}$ be independent random variables with gamma-distributions with parameters $(\gamma_{1,1} N,1)$ and $(\gamma_{1,2} N,1)$, respectively.
Using a representation
$$
B_N\eqdistr \frac{C_N^{(1)}}{C_N^{(1)}+C_N^{(2)}},
$$
for the beta-distribution, see~\cite[Theorem 3]{Jambunathan:1964}, and the central limit theorem for $C_N^{(1)}$ and $C_N^{(2)}$, it is easy to check that
\begin{equation}\label{eq:CLT_B_N}
\sqrt{N}\left(B_N-\frac{\gamma_{1,1}}{\gamma_{1,1}+\gamma_{1,2}}\right)~\todistr~\Normal(0,{\tt s}^2),\quad N\to\infty,
\end{equation}
where ${\tt s}^2:=\frac{\gamma_{1,1}\gamma_{1,2}}{(\gamma_{1,1}+\gamma_{1,2})^3}$. Moreover, let $(U_k)_{k\in\mathbb{N}}$ be a sequence of independent copies of a random variable with the uniform distribution on $[0,1]$. It is well-known that
\begin{equation}\label{eq:flt_bridge}
\left(\sqrt{N}\left(\frac{1}{N}\sum_{i=1}^{N}\1_{\{U_i\leq t\}}-t\right)\right)_{t\in [0,1]}~\Longrightarrow~ (W(t))_{t\in [0,1]},\quad N\to\infty,
\end{equation}
in the Skorokhod $J_1$-topology on $D[0,1]$, where $(W(t))_{t\in [0,1]}$ is a standard Brownian bridge.
By independence and our assumption $k(N)/N~\to~\alpha$, we have a joint convergence
\begin{multline}\label{eq:joint}
\left(\left(\sqrt{k(N)}\left(\frac{1}{k(N)}\sum_{i=1}^{k(N)}\1_{\{U_i\leq t\}}-t\right)\right)_{t\in [0,1]},B_N,\sqrt{k(N)}\left(B_N-a_1\right)\right)\\
~\Longrightarrow~ ((W(t))_{t\in [0,1]},a_1,\Normal(0,\alpha {\tt s}^2)),\quad n\to\infty,
\end{multline}
on $D[0,1]\times \R\times \R$ endowed with the product topology. Applying the map
$$
D[0,1]\times \R\times \R\ni (f(\cdot),x,y)\longmapsto f(x)+y\in \R,
$$
which is a.s.~continuous at the limiting point in~\eqref{eq:joint} we arrive at
$$
\frac{Y_{N,k(N)}-a_1 k(N)}{\sqrt{k(N)}}~\todistr~W(a_1)+\Normal(0,\alpha {\tt s}^2),\quad N\to\infty,
$$
which is equivalent to~\eqref{eq:Y_n_clt}, since ${\rm Var}\,(W(a_1))=a_1(1-a_1)$. Fix $\varepsilon\in (0,\min(a_1,1-a_1))$ and note that
\begin{align*}
\mathbb{P}(|Y_{N,k(N)}-a_1 k(N)|\geq \varepsilon k(N))&=\mathbb{P}(|Y_{N,k(N)}-\mathbb{E} Y_{N,k(N)}|\geq \varepsilon k(N))\\
&\leq \frac{{\rm Var}(Y_{N,k(N)})}{\varepsilon^2 (k(N))^2}\\
&=\frac{k(N)\E B_N(1-B_N)+N^{-1}k^2(N){\rm Var}(\sqrt{N}B_N)}{\varepsilon^2 (k(N))^2}=O((k(N)^{-1}),
\end{align*}
where the penultimate equality follows from the law of total variance, and the last equality is a consequence of $\lim_{N\to\infty}{\rm Var}(\sqrt{N}B_N)={\tt s}^2$; see~\eqref{eq:CLT_B_N}. The above estimate in conjunction with a standard tail estimate for the normal law implies that~\eqref{eq:lclt_Y_Nk} is equivalent to
\begin{equation}\label{eq:Y_n_lclt}
\sqrt{k(N)}\sup_{\ell:\,|\ell-a_1 k(N)|\leq \varepsilon k(N)}\left|\mathbb{P}(Y_{N,k(N)}=\ell)-\varphi(\ell;a_1 k(N),\sigma_{\alpha,1}^2 k(N))\right|~\longrightarrow~ 0,\quad N\to\infty.
\end{equation}
The latter can be deduced from the following explicit formula
\begin{align}\label{eq:beta_binomial_explicit}
\mathbb{P}(Y_{N,k(N)}=\ell)&=\binom{k(N)}{\ell}\frac{B(\ell+\gamma_{1,1} N,k(N)-\ell+\gamma_{1,2}N)}{B(\gamma_{1,1} N,\gamma_{1,2} N)}\\
&=\frac{k(N)+1}{B(\ell+1,k(N)-\ell+1)}\frac{B(\ell+\gamma_{1,1} N,k(N)-\ell+\gamma_{1,2}N)}{B(\gamma_{1,1} N,\gamma_{1,2} N)},\quad \ell=0,\ldots,k(N),
\end{align}
together with the asymptotic relation
\begin{equation}\label{eq:beta_uniform}
B(k(N)x,k(N)y)~=~(1+o(1))\sqrt{\frac{2\pi}{k(N)}}\sqrt{\frac{x+y}{xy}}\left(\frac{x^x y^y}{(x+y)^{x+y}}\right)^{k(N)},\quad N\to\infty,
\end{equation}	
for the beta-function, which is uniform in $x,y\in [\delta,\delta^{-1}]$, for every fixed $\delta\in (0,1)$. Out choice of $\varepsilon$ ensures that all the arguments of the beta-functions in~\eqref{eq:beta_binomial_explicit} lie in $[\delta k(N),\delta^{-1}k(N)]$ for a sufficiently small $\delta>0$. The uniform asymptotic relation~\eqref{eq:beta_uniform} is a consequence of Stirling's formula with a uniform estimate of the remainder; see, for example, Eq.~(5.11.10) and~(5.11.11) in~\cite{olver2010nist}. The proof of Proposition~\ref{prop:to_zero_in_tv} is complete.
\end{proof}

By combining \eqref{eq:distribution_Pk} and Corollary~\ref{cor:clt_llt_h_neq_0-beta-in-0-1} with Proposition~\ref{prop:to_zero_in_tv} applied for $\Theta_N=\mathcal P_N$, $K=1$, $a_1=(1+{\tt m}(\beta,h))/2$ and $\sigma_1^2:={\tt v}^2_{\beta,h}/4$, we obtain the next theorem. Note that $4a_1(1-a_1)=1-{\tt m}^2_{\beta,h}<{\tt v}^2_{\beta,h}=4\sigma_1^2$ in view of~\eqref{eq:v_beta_h_def}. Also some simple algebra yields
\begin{equation}\label{eq:gammas_curie_weiss}
\gamma_1(\beta,h):=\gamma_{1,1}=\frac{1-\beta(1-{\tt m}^2(\beta,h))}{2\beta(1-{\tt m}(\beta,h))},\quad \gamma_2(\beta,h):=\gamma_{1,2}=\frac{1-\beta(1-{\tt m}^2(\beta,h))}{2\beta(1+{\tt m}(\beta,h))}.
\end{equation}

\begin{theorem}\label{thm:P_Nk_approx_h_neq_0-beta-in-0-1}
Assume that $h\neq 0$ or $0<\beta<1$. Suppose that $k=k(N)$ satisfies~\eqref{eq:K_N_linear}. Then
$$
\lim_{N\to\infty}d_{TV}\left(\mu_N\circ\mathcal{P}_{k(N)}^{-1},\Bin\left(k(N),\BETA\left(\gamma_1(\beta,h)N,\gamma_2(\beta,h)N\right)\right)\right)=0,
$$
where $\gamma_1(\beta,h)$ and $\gamma_2(\beta,h)$ are given in~\eqref{eq:gammas_curie_weiss}.
\end{theorem}
\begin{remark} The theorem above describes the residual dependence in the propagation of chaos phenomenon via a beta-binomial distribution, which in turn has a simple interpretation as follows.
Consider a P\'{o}lya urn which initially contains $\gamma_1(\beta,h)N$ positive spins (white balls) and $\gamma_2(\beta,h)N$ negative spins (black balls). Balls are drawn one at a
time and immediately returned to the urn together with a new ball of the same color. The number of white balls drawn after $k(N)$ trials has the beta-binomial distribution $
\Bin\left(k(N),\BETA\left(\gamma_1(\beta,h)N,\gamma_2(\beta,h)N\right)\right)$. Thus, Theorem~\ref{thm:P_Nk_approx_h_neq_0-beta-in-0-1} tells us that the number of positive spins
under the Gibbs measure is close in distribution to the number of white balls drawn from the above P\'{o}lya urn after $k(N)$ trials.
\end{remark}

Similarly, Corollary~\ref{cor:clt_llt_h=0-beta>1} and Proposition~\ref{prop:to_zero_in_tv} applied with $K=2$, $p_1=p_2=1/2$ yield the next result.

\begin{theorem}\label{thm:P_Nk_approx_h=0-beta>1}
Assume that $h=0$ and $\beta>1$. Suppose that $k=k(N)$ satisfies~\eqref{eq:K_N_linear}. Then
$$
\lim_{N\to\infty}d_{TV}\left(\mu_N\circ\mathcal{P}_{k(N)}^{-1},\Bin\left(k(N),\frac{1}{2}\left(\BETA(\gamma_1(\beta,h)N,\gamma_2(\beta,h)N)+\BETA(\gamma_2(\beta,h)N,\gamma_1(\beta,h)N)\right)\right)\right)=0.
$$
\end{theorem}

\begin{proof}[Proof of Theorem~\ref{thm:main1}]
We shall treat three cases separately.

\vspace{2mm}
\noindent
{\sc Case $h=0$ and $\beta\in (0,1)$.} In this case ${\tt m}(\beta,h)={\tt m}(\beta,0)=0$ and
$\gamma_1(\beta,0)=\gamma_2(\beta,0)=(1-\beta)/(2\beta)$. In view of Theorem~\ref{thm:P_Nk_approx_h_neq_0-beta-in-0-1} it suffices to check that
\begin{equation}\label{eq:thm_proof_claim1}
\lim_{N\to\infty}d_{TV}\left(\Bin\left(k(N),\BETA\left(\gamma_1(\beta,h)N,\gamma_2(\beta,h)N\right)\right),\Bin(k(N),1/2)\right)=D\left(\frac{1-\beta+\alpha\beta}{4(1-\beta)},\frac{1}{4}\right).
\end{equation}
From~\eqref{eq:lclt_Y_Nk} we infer
$$
\lim_{N\to\infty}\sqrt{k(N)}\sup_{\ell\in\N}\left|\Bin\left(k(N),\BETA\left(\gamma_1(\beta,h)N,\gamma_2(\beta,h)N\right)\right)(\{\ell\})-\varphi\left(\ell;k(N)/2,\frac{1-\beta+\alpha\beta}{4(1-\beta)}k(N)\right)\right|=0,
$$
where the formula for the limiting variance follows from~\eqref{eq:variance_alpha}. By the classical de Moivre-Laplace local limit theorem
$$
\lim_{N\to\infty}\sqrt{k(N)}\sup_{\ell\in\N}\left|{\rm Bin}\left(k(N),1/2\right)(\{\ell\})-\varphi\left(\ell;k(N)/2,k(N)/4\right)\right|=0.
$$
Equality~\eqref{eq:thm_proof_claim1} now follows from Proposition~\ref{prop:diff_variance} below.

\vspace{2mm}
\noindent
{\sc Case $h=0$ and $\beta>1$.} By the same reasoning as in the previous case, but applying Theorem~\ref{thm:P_Nk_approx_h=0-beta>1} instead of Theorem~\ref{thm:P_Nk_approx_h_neq_0-beta-in-0-1}, we see that it suffices to prove
\begin{align}\label{eq:thm_proof_claim2}
&\hspace{-1cm}\lim_{N\to\infty}d_{TV}\Big(\Bin\left(k(N),\frac{1}{2}\left(\BETA(\gamma_1(\beta,0)N,\gamma_2(\beta,0)N)+\BETA(\gamma_2(\beta,0)N,\gamma_1(\beta,0)N)\right)\right),\notag\\
&\hspace{6cm}\Bin\left(k(N),\frac{1}{2}(\delta_{(1+{\tt m}(\beta,0))/2}+\delta_{(1-{\tt m}(\beta,0))/2})\right)\Big)\notag\\
&=D\left(\frac{1-{\tt m}^2(\beta,0)}{4}\left(1+\frac{\alpha\beta(1-{\tt m}^2(\beta,0))}{1-\beta(1-{\tt m}^2(\beta,0))}\right),\frac{1-{\tt m}^2(\beta,0)}{4}\right).
\end{align}
This again follows from~\eqref{eq:lclt_Y_Nk} but now with $K=2$, $p_1=p_2=1/2$ and the de Moivre-Laplace local limit theorem in conjunction with Proposition~\ref{prop:diff_variance_mixture} below. Put
$$
\mathfrak{s}^2(\beta,h):=\frac{1-{\tt m}^2(\beta,h)}{4}\left(1+\frac{\alpha\beta(1-{\tt m}^2(\beta,h))}{1-\beta(1-{\tt m}^2(\beta,h))}\right),
$$
and note that in~\eqref{eq:lclt_Y_Nk} the variances coincide:
$$
\sigma_{\alpha,1}^2=\sigma_{\alpha,2}^2=\mathfrak{s}^2(\beta,0).
$$
The second equality here follows upon substituting~\eqref{eq:gammas_curie_weiss} into~\eqref{eq:variance_alpha} and using that $\gamma_{2,1}=\gamma_{1,2}=\gamma_1(\beta,0)$, $\gamma_{1,1}=\gamma_{2,2}=\gamma_2(\beta,0)$ in this case.

\vspace{2mm}
\noindent
{\sc Case $h\neq 0$ and $\beta>0$.} In this case we again apply Theorem~\ref{thm:P_Nk_approx_h_neq_0-beta-in-0-1}. From~\eqref{eq:lclt_Y_Nk} it follows that
\begin{multline*}
\lim_{N\to\infty}\sqrt{k(N)}\sup_{\ell\in\N}\left|\Bin\left(k(N),\BETA\left(\gamma_1(\beta,h)N,\gamma_2(\beta,h)N\right)\right)(\{\ell\})-\varphi\left(\ell;\frac{1+{\tt m}(\beta,h)}{2}k(N),\mathfrak{s}^2(\beta,h)k(N)\right)\right|\\
=0.
\end{multline*}
Combining this with the de Moivre-Laplace local limit theorem and using Proposition~\ref{prop:diff_variance} we arrive at~\eqref{eq:thm1_claim3}.
\end{proof}

\begin{remark}
Here we give a comparison of our findings with some known results on the uniform distribution on a high-dimensional sphere, which bears some similarities with the Curie-Weiss model in the context of the propagation of chaos phenomenon.
For every $N\in \N$, consider a random vector $\xi_N = (\xi_{1;N},\ldots, \xi_{N;N})$ which is uniformly distributed on the sphere $\sqrt{N} \mathbb S^{N-1}$ of radius $\sqrt N$ in $\R^N$. It is a classical result of Maxwell-Poincar\'{e}-Borel that, for every $k\in \N$, the distribution of any $k$ components of $\xi_N$ converges weakly to a $k$-dimensional standard normal distribution, as $N\to\infty$. Moreover, if $k=k(N)$ depends on $N$ such that $k(N)/N\to 0$, then it has been shown in~\cite[Section~2]{diaconis_freedman:1987} that the total variation distance between the distribution of $(\xi_{1;n},\ldots, \xi_{k; N})$ and the standard normal distribution on $\mathbb R^{k}$ converges to $0$. On the other hand, for $k(N)/N\to \alpha$ with $\alpha \neq 0$, the total variation distance between these distributions converges to a non-zero limit which has been identified in~\cite[Theorem~1.6~(b)]{diaconis_freedman:1988}; see also Eq.~(2.12) on p.~403 in~\cite{diaconis_freedman:1987}. All these results are similar to what we know about the Curie-Weiss model. There is, however, one important difference: An approximation by a variance mixture of normal distributions is not possible if $\alpha >0$ in the setting of $\xi_N$. A much more general result has been shown in~\cite[Theorem~2.3~(b)]{diaconis_freedman:1988}. Let us give a short informal argument.  Let $(\zeta_1,\ldots, \zeta_{k})$ be a random vector with the standard normal distribution on $\mathbb R^{k}$ and let $R>0$ be a mixing variable independent of the $\zeta_i$'s. We ask whether it is possible to approximate $(\xi_{1;N},\ldots, \xi_{k;N})$ by $R\cdot (\zeta_1,\ldots, \zeta_k)$ in the total variation distance. By rotation invariance, it suffices to consider the distance between the squared radial parts; see~ Eq.~(2.4) on p.~402 in~\cite{diaconis_freedman:1987}. Since $\frac 1N (\xi_{1;N}^2 + \cdots + \xi_{k;N}^2)$ is beta-distributed with parameters $(k/2,(N-k)/2)$, it follows that
$$
\mathbb E (\xi_{1;N}^2 + \cdots + \xi_{k;N}^2) = k,
\qquad
\mathrm{Var} (\xi_{1;N}^2 + \cdots + \xi_{k;N}^2) = \frac{k(N-k)/4}{(N/2)^2 (N/2 + 1)} N^2 \sim 2 \alpha (1-\alpha) N.
$$
On the other hand, the squared radial part of $R(\zeta_1,\ldots, \zeta_{k})$ is $R^2\chi_k^2$, where $\chi_k^2:= \zeta_1^2 + \ldots + \zeta_k^2$ has a chi-square distribution with $k$ degrees of freedom, and we have
$$
\mathbb E (R^2 (\zeta_{1}^2 + \cdots + \zeta_{k}^2)) = \E (R^2) k,
\qquad
\mathrm{Var} (R^2 (\zeta_{1}^2 + \cdots + \zeta_{k}^2)) = \E (R^4)(2k+k^2) - (\E (R^2) k)^2. 
$$
where we used that $\E (\chi_k^4)= 2k+k^2$. If we want to match expectations, we need $\E(R^2) = 1$, but then $\E (R^4)(2k+k^2) - (\E (R^2) k)^2 \geq 2k + k^2 - k^2 = 2k \sim 2\alpha N$,  and we cannot match the variances since $2\alpha > 2\alpha (1-\alpha)$. This is in sharp contrast to the case of the Curie-Weiss model, for which Theorem \ref{thm:P_Nk_approx_h_neq_0-beta-in-0-1} shows that the larger variance of $\mathcal P_{k(N)}$ can be artificially matched by a mixed binomial distribution.

\end{remark}

\section{Proof of Proposition~\ref{prop:clt_llt_h=0-beta>1}}\label{sec:proof_magnetization}
Recall that we work under the assumptions $h=0$ and $\beta>1$, which imply ${\tt m}:={\tt m}(\beta,0)\in (0,1)$. For simplicity we also assume throughout this proof that $N=2n$ is even. The case of odd $N$ can be treated similarly. For every $-n\leq \ell\leq n$, we have
\begin{equation}\label{eq:111}
\mu_N(\{\sigma:Nm_N(\sigma)=2\ell\})= \frac{|\{\sigma:Nm_N(\sigma)=2\ell\}|}{Z_{N}(\beta)}\exp\left(\frac{\beta}{2N}(2\ell)^2\right)=
\frac{1}{Z_{2n}(\beta)}\binom{2n}{n+\ell}\exp(\beta\ell^2/n).
\end{equation}
For further use we record the asymptotic formula for the partition function: if $\beta>1$, then
\begin{equation}\label{eq:partition_function_asymp_beta>1}
Z_N(\beta)~=~(1+o(1))\frac{2{\tt v}_{\beta,0}}{\sqrt{1-{\tt m}^2}}2^{N}\eee^{-N(\mathcal{I}({\tt m})-\beta {\tt m}^2/2)},\quad N\to\infty,
\end{equation}
where $\mathcal{I}(x):=\frac{1}{2}((1+x)\log(1+x)+(1-x)\log(1-x))$ for $x\in [-1,1]$, which can be derived
from Th\'eor\`eme B in~\cite{BB90} by specializing to the Curie-Weiss model or from (3.19) in~\cite{Kabluchko+Loewe+Schubert:2022} by
choosing $p=1$; see also~\cite[Theorem 1.3]{Shamis+Zeitouni:2018}.
For every fixed $n\in\N$, the function
\begin{equation}\label{eq:pmf_func}
\{-n,\ldots,n\}\ni\ell\longmapsto \binom{2n}{n+\ell}\exp(\beta\ell^2/n)
\end{equation}
is even and attains two local maxima at $\ell\approx n{\tt m}$ and $\ell\approx -n{\tt m}$. This can be checked by calculating the ratio of its two consecutive values. Using the aforementioned symmetry and the standard tail estimate for the normal density, we see that~\eqref{eq:clt_llt_h=0-beta>1} is equivalent to
$$
\lim_{n\to\infty}\sqrt{N}\sup_{\ell\in \mathbb{Z}_{\geq 0}}\left|\mu_N\left(\left\{\sigma: Nm_N(\sigma)=2\ell\right\}\right)-\frac{1}{2}\varphi\left(\ell;\frac{N{\tt m}}{2},\frac{N{\tt v}^2_{\beta,0}}{4}\right)\right|=0,
$$
or, by~\eqref{eq:111}, to
\begin{equation}\label{eq:clt_llt_h=0-beta>1_proof}
\lim_{n\to\infty}\sqrt{n}\sup_{\ell\in \mathbb{Z}_{\geq 0}}\left|\frac{1}{Z_{2n}(\beta)}\binom{2n}{n+\ell}\exp(\beta\ell^2/n)-\frac{1}{2}\varphi\left(\ell;n{\tt m},\frac{n{\tt v}^2_{\beta,0}}{2}\right)\right|=0.
\end{equation}
Put $\mathcal{C}_n:=\{j\in\mathbb{Z}_{\geq 0}:j\in [{\tt m}n-n^{7/12},{\tt m}n+n^{7/12}]\}$. By Stirling's formula, uniformly in $\ell\in \mathcal{C}_n$, it holds
$$
2^{-2n}\binom{2n}{n+\ell}=(1+o(1))\sqrt{\frac{1}{\pi n}}\frac{1}{\sqrt{1-{\tt m}^2}}\eee^{-2n\mathcal{I}(\ell/n)}.
$$
Combining this with~\eqref{eq:partition_function_asymp_beta>1} we deduce that, uniformly in $\ell\in \mathcal{C}_n$,
\begin{equation*}
\frac{1}{Z_{2n}(\beta)}\binom{2n}{n+\ell}\exp(\beta\ell^2/n)=\frac{1+o(1)}{2{\tt v}_{\beta,0}\sqrt{\pi n}}\exp\left(-2n\left(\mathcal{I}(\ell/n)-\mathcal{I}({\tt m})+\frac{\beta{\tt m}^2}{2}-\frac{\beta \ell^2}{2n^2}\right)\right).
\end{equation*}
By the definition of ${\tt m}$ the first derivative of $t\mapsto \mathcal{I}(t)-\beta t^2/2$ vanishes at $t={\tt m}$ and the second derivative at $t={\tt m}$ is equal to ${\tt v}^{-2}_{\beta,0}$. Thus, plugging the Taylor expansion
$$
\mathcal{I}(\ell/n)=\mathcal{I}({\tt m})+\mathcal{I}'({\tt m})\frac{\ell-{\tt m}n}{n}+\frac{\mathcal{I}''({\tt m})}{2}\left(\frac{\ell-{\tt m}n}{n}\right)^2+ O\left(\left(\frac{\ell-{\tt m}n}{n}\right)^3\right),
$$
we obtain
$$
\frac{1}{Z_{2n}(\beta)}\binom{2n}{n+\ell}\exp(\beta\ell^2/n)~=~\frac{(1+o(1))}{2{\tt v}_{\beta,0}\sqrt{\pi n}}\exp\left(-\frac{(\ell -{\tt m}n)^2}{{\tt v}^2_{\beta,0}n}\right)=(1+o(1))\frac{1}{2}\varphi\left(\ell;n{\tt m},\frac{n{\tt v}^2_{\beta,0}}{2}\right),
$$
which is again uniform over $\ell\in\mathcal{C}_n$. This yields that~\eqref{eq:clt_llt_h=0-beta>1_proof} holds with $\sup_{\ell\in\mathbb{Z}_{\geq 0}}$ replaced by $\sup_{\ell\in\mathbb{Z}_{\geq 0},\ell\in \mathcal{C}_n}$.

To check that the indices outside $\mathcal{C}_n$ are negligible we use the same reasoning as in~\cite{Kabluchko+Loewe+Schubert:2022}; see pp.~548-549 therein. According to the estimate~(3.8) in~\cite{Kabluchko+Loewe+Schubert:2022}
$$
2^{-2n}\binom{2n}{2n+\ell}\leq \exp(-2n\mathcal{I}(\ell/n)),\quad -n\leq \ell \leq n.
$$
Therefore, for some absolute constant $C_1>0$ all $n\in\N$ and $|\ell|\leq n$,
$$
\frac{1}{Z_{2n}(\beta)}\binom{2n}{n+\ell}\exp(\beta\ell^2/n)\leq C_1\exp\left(-2n\left(\mathcal{I}(\ell/n)-\mathcal{I}({\tt m})+\frac{\beta{\tt m}^2}{2}-\frac{\beta \ell^2}{2n^2}\right)\right).
$$
The function $t\mapsto \mathcal{I}(t)-\beta t^2/2$ attains a local minimum on the positive half-line at ${\tt m}$. Therefore, for $|\ell-{\tt m}n|\geq n^{7/12}$, $|\ell|\leq n$, and some $C_2>0$
$$
\mathcal{I}(\ell/n)-\mathcal{I}({\tt m})+\frac{\beta{\tt m}^2}{2}-\frac{\beta \ell^2}{2n^2}\geq C_2\left(\frac{(\ell-{\tt m}n)^2}{n^2}\right)\geq C_2 n^{-5/6}.
$$
Thus, for $|\ell-{\tt m}n|\geq n^{7/12}$ and $|\ell|\leq n$,
$$
\frac{1}{Z_{2n}(\beta)}\binom{2n}{n+\ell}\exp(\beta\ell^2/n)\leq C_1\eee^{-2C_2 n^{1/6}}.
$$
Combing this with a standard tail estimate for the normal distribution shows that~\eqref{eq:clt_llt_h=0-beta>1_proof} holds with $\sup_{\ell\in\mathbb{Z}_{\geq 0}}$ replaced by $\sup_{\ell\in\mathbb{Z}_{\geq 0},\ell\notin \mathcal{C}_n}$. The proof of Proposition~\ref{prop:clt_llt_h=0-beta>1} is complete.

\section{Auxiliary results}\label{sec:appendix}
\begin{proposition}\label{prop:diff_variance}
Assume that two sequences of integer-valued random variables $(Q^{(1)}_N)_{N\in\mathbb{N}}$ and $(Q^{(2)}_N)_{N\in\mathbb{N}}$ satisfy the local limit theorems
\begin{equation}\label{eq:lclt_assumption1}
\lim_{N\to\infty}\sqrt{N}\sup_{\ell\in\mathbb{Z}}\left|\mathbb{P}(Q^{(i)}_N=\ell)-\varphi(\ell;N {\tt m},N{\tt v}_i^2)\right|=0,\quad i=1,2,
\end{equation}
with the same mean ${\tt m}$ and arbitrary variances ${\tt v}_1^2$ and ${\tt v}_2^2$. Then
$$
\lim_{N\to\infty}d_{TV}(\mathcal{L}(Q^{(1)}_N),\mathcal{L}(Q^{(2)}_N))=D({\tt v}^2_1,{\tt v}^2_2),
$$
 where the right-hand side was defined in~\eqref{eq:def_D_dist_normal_distr}.
In particular, if ${\tt v}_1^2={\tt v}_2^2$, then 
$$
\lim_{N\to\infty}d_{TV}(\mathcal{L}(Q^{(1)}_N),\mathcal{L}(Q^{(2)}_N))=0.
$$
\end{proposition}
\begin{proof}
Without loss of generality assume that ${\tt m}=0$. Fix $\varepsilon>0$ and $M>0$. Then, there exists $N_0(\varepsilon)$ such that
$$
\left|\mathbb{P}(Q^{(i)}_N=\ell)-\varphi(\ell;0,N{\tt v}_i^2)\right|\leq \frac{\varepsilon}{\sqrt{N}},\quad i=1,2,\quad N\geq N_0(\varepsilon),\quad -M\sqrt{N}\leq \ell \leq M\sqrt{N}.
$$
Using~\eqref{eq:dtv_sum} we infer
\begin{multline*}
\left|2d_{TV}(\mathcal{L}(Q^{(1)}_N),\mathcal{L}(Q^{(2)}_N))-\sum_{\ell:\,|\ell|\leq M\sqrt{N}}\left|\varphi(\ell;0,N{\tt v}_1^2)-\varphi(\ell;0,N{\tt v}_2^2)\right|\right|\\
\leq \mathbb{P}(|Q^{(1)}_N|> M\sqrt{N})+\mathbb{P}(|Q^{(2)}_N|> M\sqrt{N})+4M\varepsilon+\sum_{\ell:\,|\ell|> M\sqrt{N}}\varphi(\ell;0,N{\tt v}_1^2)+\sum_{\ell:\,|\ell|> M\sqrt{N}}\varphi(\ell;0,N{\tt v}_2^2).
\end{multline*}
Note that
\begin{align}\label{eq:riemann}
\lim_{N\to\infty}\sum_{\ell:\,|\ell|\leq M\sqrt{N}}\left|\varphi(\ell;0,N{\tt v}_1^2)-\varphi(\ell;0,N{\tt v}_2^2)\right|&=\lim_{N\to\infty}\frac{1}{\sqrt{N}}\sum_{\ell:\,|\ell|\leq M\sqrt{N}}\left|\varphi(N^{-1/2}\ell;0,{\tt v}_1^2)-\varphi(N^{-1/2}\ell;0,{\tt v}_2^2)\right|\nonumber	\\
&=\int_{-M}^{M}\left|\varphi(t;0,{\tt v}_1^2)-\varphi(t;0,{\tt v}_2^2)\right|{\rm d}t,
\end{align}
since the sum on the left-hand side is a Riemann sum of a Riemann integrable function on $[-M,M]$. Thus,
\begin{align*}
&\hspace{-1cm}\limsup_{N\to\infty}\left|2d_{TV}(\mathcal{L}(Q^{(1)}_n),\mathcal{L}(Q^{(2)}_n))-\int_{-M}^{M}\left|\varphi(t;0,{\tt v}_1^2)-\varphi(t;0,{\tt v}_2^2)\right|{\rm d}t\right|\\
&\leq \int_{|t|>M}\varphi(t;0,{\tt v}_1^2){\rm d}t+\int_{|t|>M}\varphi(t;0,{\tt v}_2^2){\rm d}t+4M\varepsilon\\
&\hspace{1cm}+\limsup_{N\to\infty}N^{-1/2}\sum_{\ell:\,|\ell|> M\sqrt{N}}\varphi(N^{-1/2}\ell;0,N{\tt v}_1^2)+\limsup_{N\to\infty}N^{-1/2}\sum_{\ell:\,|\ell|> M\sqrt{N}}\varphi(N^{-1/2}\ell;0,{\tt v}_2^2)\\
&= 2\int_{|t|>M}\varphi(t;0,{\tt v}_1^2){\rm d}t+2\int_{|t|>M}\varphi(t;0,{\tt v}_2^2){\rm d}t+4M\varepsilon,
\end{align*}
where for the last equality we used that $t\mapsto \varphi(t;0,v_i^2)$, $i=1,2$, are directly Riemann integrable on $\mathbb{R}$. Sending $\varepsilon\to 0$ and then $M\to\infty$ completes the proof.
\end{proof}

For the mixtures of normal densities we have the following generalization of Proposition~\ref{prop:diff_variance}.

\begin{proposition}\label{prop:diff_variance_mixture}
If we replace~\eqref{eq:lclt_assumption1} by
\begin{equation}\label{eq:lclt_assumption1_mixture}
\lim_{N\to\infty}\sqrt{N}\sup_{\ell\in\mathbb{Z}}\left|\mathbb{P}(Q^{(i)}_N=\ell)-\sum_{j=1}^{K}p_j\varphi(\ell;N {\tt m}_j,N{\tt v}_{i,j}^2)\right|=0,\quad i=1,2,
\end{equation}
for some $K\in\mathbb{N}$, a collection of positive weights $p_1,\ldots,p_K$ satisfying $\sum_{j=1}^{K}p_j=1$, ${\tt m}_j\in \mathbb{R}$ such that ${\tt m}_i\neq {\tt m}_j$, $i\neq j$, and ${\tt v}^2_{i,j}>0$, $j=1,\ldots,K$, then
$$
\lim_{N\to\infty}d_{TV}(\mathcal{L}(Q^{(1)}_N),\mathcal{L}(Q^{(2)}_N))=\sum_{j=1}^{K}p_j D({\tt v}^2_{1,j},{\tt v}^2_{2,j}).
$$
\end{proposition}
\begin{proof}
We shall sketch the proof omitting details which are similar to those used in the proof of Proposition~\ref{prop:diff_variance}. Fix large positive constants $M_1,\ldots,M_K$. Split the range of summation on the right-hand side of equality
$$
d_{TV}(\mathcal{L}(Q^{(1)}_N),\mathcal{L}(Q^{(2)}_N))=\frac{1}{2}\sum_{\ell\in\mathbb{Z}}\left|\mathbb{P}(Q^{(1)}_N=\ell)-\mathbb{P}(Q^{(2)}_N=\ell)\right|
$$
into $K+1$ sets
$$
A_j(N):=\{\ell\in\mathbb{Z}: |\ell-{\tt m}_j N|\leq M_j\sqrt{N}\},\quad j=1,\ldots,K,\quad A_{K+1}(N):=\mathbb{Z}\setminus \left(\cup_{j=1}^{K}A_j(N)\right).
$$
These sets are pairwise disjoint for all large enough $N$, since ${\tt m}_j$'s are pairwise different. In view of~\eqref{eq:lclt_assumption1_mixture} and by the standard tail estimate for the normal law, for every $j=1,\ldots,K$,
$$
\frac{1}{2}\sum_{\ell\in A_j(N)}\left|\mathbb{P}(Q^{(1)}_N=\ell)-\mathbb{P}(Q^{(2)}_N=\ell)\right|=\frac{p_j}{2}\sum_{\ell\in A_j(N)}\left|\varphi(\ell;N {\tt m}_j,N{\tt v}_{1,j}^2)-\varphi(\ell;N {\tt m}_j,N{\tt v}_{2,j}^2)\right|+r_j(N,M_j),
$$
where the remainders satisfy $\lim_{M_j\to\infty}\limsup_{N\to\infty}|r_j(N,M_j)|=0$, $j=1,\ldots,K$. Furthermore,
\begin{multline*}
\lim_{\min_j M_j\to\infty}\limsup_{N\to\infty}\sum_{\ell\in A_{K+1}(N)}\left|\mathbb{P}(Q^{(1)}_N=\ell)-\mathbb{P}(Q^{(2)}_N=\ell)\right|\\
\leq \lim_{\min_j M_j\to\infty}\limsup_{N\to\infty}\left(\mathbb{P}(Q^{(1)}_N\in A_{K+1}(N))+\mathbb{P}(Q^{(2)}_N\in A_{K+1}(N))\right)=0.
\end{multline*}
It remains to note that~\eqref{eq:riemann} yields
$$
\lim_{M_j\to\infty}\lim_{N\to\infty}\frac{1}{2}\sum_{\ell\in A_j(N)}\left|\varphi(\ell;N {\tt m}_j,N{\tt v}_{1,j}^2)-\varphi(\ell;N {\tt m}_j,N{\tt v}_{2,j}^2)\right|=D({\tt v}_{1,j}^2,{\tt v}_{2,j}^2),\quad j=1,\ldots,K.
$$
\end{proof}

%\bibliographystyle{abbrv}
%\bibliography{LiteraturDatenbank}

\end{document}